\documentclass[a4paper,11pt]{amsart}
\usepackage[utf8]{inputenc}
\usepackage{amsmath,amssymb,amsfonts}
\usepackage{epsfig}
\usepackage{mathrsfs}
\usepackage{graphicx,color}
\usepackage{eucal}
\usepackage{bbm}

\newtheorem{thm}{Theorem}[section]
\newtheorem{prop}{Proposition}[section]

\newtheorem{rmk}{Remark}[section]
\newtheorem{lma}{Lemma}[section]

\newtheorem{exm}{Example}[section]

\newcommand{\norm}[1]{\left| \! \left| #1\right| \!\right|}
\newcommand{\bnorm}[1]{\big| \! \big| #1\big| \!\big|}
\newcommand{\Xs}{\mathcal{X}}
\newcommand{\sX}{\mathcal{X}}
\newcommand{\Ys}{\mathcal{Y}}
\newcommand{\sY}{\mathcal{Y}}

\newcommand{\sU}{\mathcal{U}}
\newcommand{\Ch}{C_{\square}}

\newcommand{\Ex}[1]{\mathbb{E} \! \left( #1 \right)}
\newcommand{\ip}[1]{\left< #1 \right>}
\newcommand{\dom}{\mathscr{D}}
\newcommand{\Cov}[1]{\textup{Cov}\left[ #1 \right]}
\newcommand{\tr}{\textup{tr}}
\newcommand{\eig}{\textup{eig}}

\newcommand{\sumk}{\sum_{k=1}^{\infty}}

\numberwithin{equation}{section}
\font\eka=cmex10
\usepackage{ae}
\def\ind{\mathrel{\hbox{\rlap{%
\hbox to 7.5pt{\hrulefill}}\raise6.6pt\hbox{\eka\char'167}}}}
\begin{document}

\title[]{Convergence of discrete-time Kalman filter estimate to continuous-time estimate for systems with unbounded observation*\footnote{*M\lowercase{anuscript submitted for publication.}}}

\author[Atte Aalto]{Atte Aalto$^{\lowercase{\textup{a,b}}}$ \\ \\ $^{\lowercase{\textup{a}}}$D\lowercase{epartment of} M\lowercase{athematics and} S\lowercase{ystems} A\lowercase{nalysis}, A\lowercase{alto} U\lowercase{niversity}, E\lowercase{spoo}, F\lowercase{inland} \\ $^{\lowercase{\textup{b}}}$I\lowercase{nria}, U\lowercase{niversit\'e} P\lowercase{aris}--S\lowercase{aclay}, P\lowercase{alaiseau}, F\lowercase{rance}; M$\Xi$DISIM \lowercase{team} }

\thanks{Email: atte.ej.aalto@gmail.com}

\begin{abstract}
In this article, we complement recent results on the convergence of the state estimate obtained by applying the discrete-time Kalman filter on a time-sampled continuous-time system. As the temporal discretization is refined, the estimate converges to the continuous-time estimate given by the Kalman--Bucy filter. We shall give bounds for the convergence rates for the variance of the discrepancy between these two estimates. The contribution of this article is to generalize the convergence results to systems with  unbounded observation operators under different sets of assumptions, including systems with diagonalizable generators, systems with admissible observation operators, and systems with analytic semigroups. The proofs are based on applying the discrete-time Kalman filter on a dense, numerable subset on the time interval $[0,T]$ and bounding the increments obtained. These bounds are obtained by studying the regularity of the underlying semigroup and the noise-free output.

\medskip


\end{abstract}

\maketitle

\section{Introduction}

The minimum variance state estimate for linear systems with Gaussian noise processes is given by the continuous-time Kalman filter. However, for obvious reasons, in a practical implementation the continuous-time system is often first discretized, and then the discrete-time Kalman filter is used on the discretized system.
The objective of this article is to expand the recent results presented in \cite{tempI} by the author on the convergence of the state estimate given by the discrete-time Kalman filter on the sampled system to the continuous-time estimate. There convergence results were shown for finite-dimensional systems and infinite-dimensional systems with bounded observation operators. The expansion in this paper covers systems with unbounded observation operators and systems whose dynamics are governed by an analytic semigroup. In particular, we shall show convergence rate estimates for the variance of the discrepancy between the discrete- and continuous-time estimates.

We study systems whose dynamics are given by 
\begin{equation} \label{eq:system}
\begin{cases}
dz(t)=Az(t) \, dt + Bdu(t), \qquad t \in \mathbb{R}^+, \\
dy(t)=Cz(t) \, dt + dw(t), \\
z(0)=x
\end{cases}
\end{equation}
where $A:\sX \to \sX$, $B:\sU \to \sX$, and $C:\sX \to \sY$. The Hilbert spaces $\sX$,  $\sU = \mathbb{R}^q$, and $\sY=\mathbb{R}^r$ are called the \emph{state space}, the \emph{input space}, and  the \emph{output space}, respectively. The mapping $A$ is the generator of a $C_0$-semigroup $e^{At}$ on $\sX$ with domain $\dom(A)$, $B:\mathbb{R}^q \to \sX$ is the \emph{control operator},  and $C:\sX \to
\mathbb{R}^r$ is called the \emph{observation operator}. The dynamics equations \eqref{eq:system} are given in the form of stochastic differential equations, see~\cite{Oks} by \O ksendal for background. The input and output noise processes $u$ and $w$ are assumed to be $q$- and $r$-dimensional Brownian motions with incremental covariance matrices $Q \ge 0$ and $R>0$, respectively. Without loss of generality, we assume that there is no deterministic input, as it can always be removed by the usual techniques. The initial state $x \in \Xs$ is assumed to be a Gaussian random variable with mean $m$ and covariance $P_0$, denoted $x \sim N(m,P_0)$, and $u$, $w$, and $x$ are assumed to be mutually independent.

The purpose of this paper is to study the discrepancy of the discrete- and continuous-time state estimates, defined by
\begin{equation} \label{eq:estimates} \hat{z}_{T,n}:=\Ex{z(T) \,
    \Big|\left\{y \! \left(\tfrac{iT}{n}\right) \! \right\}_{i=1}^n}  \
  \textrm{ and }  \ \hat{z}(T):=\Ex{z(T) \, \big|\big\{y(s),s \le T
    \big\}}, \hspace{-4mm}
\end{equation}
respectively, and in particular, find convergence rate estimates for the variance $\Ex{\norm{\hat{z}_{T,n}-\hat{z}(T)}_{\Xs}^2}$ as $n \to \infty$ when the observation operator $C$ is not bounded, which typically occurs when we get a pointwise or a boundary measurement from the computational domain of a system whose dynamics are governed through a partial differential equation. However, we do assume that $C \in \mathcal{L}(\dom(A),\sY)$.
 
 The state estimates $\hat z_{T,n}$ and $\hat z(T)$ are obtained by the Kalman(--Bucy) filter --- provided that the continuous-time Kalman filter equations are solvable. The Kalman filter was originally presented by Kalman in \cite{Kalman} for discrete-time systems and by Kalman and Bucy in \cite{Kalman_Bucy} for continuous-time systems. The infinite-dimensional generalization has been treated for example by Falb in \cite{Falb}, by Bensoussan in \cite{Bensoussan}, by Curtain and Pritchard in \cite{Curtain_Pritchard}, and by Horowitz in \cite{Horowitz}. Of course the infinite-dimensional setting gives rise to many technical issues, such as unbounded control and observation operators and the solvability of the corresponding Riccati equations. These problems are tackled for example by Da Prato and Ichikawa in  \cite{DaPrato_Ichikawa} and by Flandoli in \cite{Flandoli86}.

  In the results of this paper we assume that the temporal discretization can be done perfectly, so that the only error source is the sampling of the continuous-time output signal. We refer to the review article \cite{Sample} by Goodwin \emph{et al}. for a discussion on the sampling of continuous-time systems and in particular \cite{Salgado} by Salgado \emph{et al.} for a study on the sampled data Riccati equations and the Kalman filter. In practice, approximative numerical schemes are used for solving both the state estimate $\hat z_{T,n}$ and the corresponding error covariance. For a discussion on this topic, see \cite{Axelsson} by Axelsson and Gustafsson and \cite{Frogerais} by Frogerais \emph{et al.} treating nonlinear systems.

 In Section~\ref{sec:background}, we shall introduce the ingredients for the proofs of our results. The main idea is to apply the discrete-time Kalman filter on a dense, numerable subset of the interval $[0,T]$. This way we obtain a martingale that starts from the discrete-time estimate $\hat z_{T,n}$ and converges almost surely to the continuous-time estimate $\hat z(T)$. We shall then find bounds for the increments of this martingale. These bounds are obtained by studying the regularity of the semigroup $e^{At}$ and in particular, the smoothness of the noise-free output $Ce^{At}x$ for $x \in \sX$. As was noted in \cite{tempI} and as seen later in the proof of Theorem~\ref{thm:noise}, the effects of the input noise process $u$ and the initial state $x$ can be treated separately. Therefore we shall first derive several results with different assumptions on the system concerning just the effect of the initial state in Section~\ref{sec:noiseless}. Finally, in Section~\ref{sec:noise}, we shall consider the effect of the input noise. The input noise effect is shown with the assumption of admissibility of the observation operator $C$.

\subsection*{Notation and standing assumptions}
We denote by $\{ e_k \}_{k=1}^{\infty} \subset \dom(A)$ an orthonormal basis for the state space $\sX$. The operator $A$ generates a strongly continuous semigroup that is  bounded by
 $\norm{e^{At}}_{\mathcal{L}(\sX)} \le \mu$ for $t \in [0,T]$.

\section{Background} \label{sec:background}

The idea of the proofs is exactly the same as in \cite{tempI}, but here we need to deal with many more technical issues. That is, we define a dense, numerable subset of the time interval $[0,T]$ and apply the discrete-time Kalman filter in this subset. Then we compute an upper bound for each increment in the state estimate and finally sum up these bounds. So let us define the time points $t_j$ for $j=1,2,...$ through the dyadic division
\begin{equation} \label{eq:times}
t_j= \left\{ \hspace{-1.5mm} \begin{array}{ll} \frac{T}nj, & j=1,...,n \vspace{1.5mm} \\ \frac{T}{2^K n} (2j-1), & j=2^{K-1}n+1,...,2^Kn, \textrm{ for some } K=1,2,... \end{array} \right.
\end{equation}
The time point definition is illustrated in Figure~\ref{fig:times}. Then define $\textup{T}_j:=\{ t_i \}_{i=1}^j$ and the $\Xs$-valued martingale $\tilde z_j=\Ex{z(T)|\{ y(t), t \in \textup{T}_j \}}$. Define also the shorthand notation $y(\textup{T}_j)=\{ y(t), t \in \textup{T}_j \}$. Now it holds that $\tilde z_n=\hat z_{T,n}$, and as discussed in \cite[Section~2.1]{tempI}, as $j \to \infty$, the martingale $\tilde z_j$ converges almost surely strongly to $\hat z(T)$. The idea in the proofs in this paper is to find upper bounds for the increments $\tilde z_{j+1}-\tilde z_j$, for $j \ge n$.

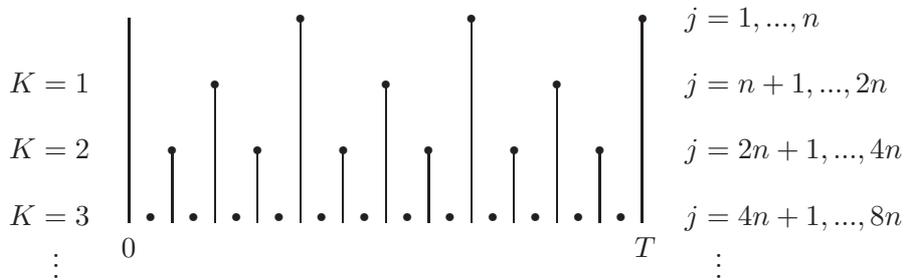
\begin{figure}[b]
\center
\begin{picture}(270,105)
\thicklines
\put(10,100){\line(0,-1){77}}
\put(202,100){\line(0,-1){77}}

\thinlines
\put(74,100){\circle*{3}}
\put(74,100){\line(0,-1){77}}
\put(138,100){\circle*{3}}
\put(138,100){\line(0,-1){77}}
\put(202,100){\circle*{3}}

\put(42,75){\circle*{3}}
\put(42,75){\line(0,-1){52}}
\put(106,75){\circle*{3}}
\put(106,75){\line(0,-1){52}}
\put(170,75){\circle*{3}}
\put(170,75){\line(0,-1){52}}

\put(26,50){\circle*{3}}
\put(26,50){\line(0,-1){27}}
\put(58,50){\circle*{3}}
\put(58,50){\line(0,-1){27}}
\put(90,50){\circle*{3}}
\put(90,50){\line(0,-1){27}}
\put(122,50){\circle*{3}}
\put(122,50){\line(0,-1){27}}
\put(154,50){\circle*{3}}
\put(154,50){\line(0,-1){27}}
\put(186,50){\circle*{3}}
\put(186,50){\line(0,-1){27}}

\put(18,25){\circle*{3}}
\put(34,25){\circle*{3}}
\put(50,25){\circle*{3}}
\put(66,25){\circle*{3}}
\put(82,25){\circle*{3}}
\put(98,25){\circle*{3}}
\put(114,25){\circle*{3}}
\put(130,25){\circle*{3}}
\put(146,25){\circle*{3}}
\put(162,25){\circle*{3}}
\put(178,25){\circle*{3}}
\put(194,25){\circle*{3}}

\put(218,97){$j = 1,...,n$}
\put(218,72){$j=n+1,...,2n$}
\put(218,47){$j=2n+1,...,4n$}
\put(218,22){$j=4n+1,...,8n$}

\put(-35,72){$K=1$}
\put(-35,47){$K=2$}
\put(-35,22){$K=3$}

\put(7,10){0}
\put(199,10){$T$}
\put(229,2){$\vdots$}
\put(-19,2){$\vdots$}

\end{picture}
\caption{Illustration of the time point addition scheme in the construction of the martingales $\tilde x_j$ and $\tilde z_j$ (from \cite{tempI}).}\label{fig:times}
\end{figure}

As the martingale $\tilde z_j$ is square integrable, we have the telescope identity shown in \cite[Lemma~1]{tempI} for $L,N \in \mathbb{N}$ with $L \ge N$:
\begin{equation} \label{eq:telescope}
\Ex{\norm{\tilde z_L-\tilde z_N}_{\Xs}^2}=\sum_{j=N}^{L-1} \Ex{\norm{\tilde z_{j+1}-\tilde z_j}_{\Xs}^2}.
\end{equation}
We remark that setting $N=n$ and letting $L \to \infty$ gives $\Ex{\norm{\hat z(T)-\hat z_{T,n}}_{\Xs}^2}=\sum_{j=n}^{\infty} \Ex{\norm{\tilde z_{j+1}-\tilde z_j}_{\Xs}^2}$.

Let us then establish an expression for one increment $\Ex{\norm{\tilde z_{j+1}-\tilde z_j}_{\Xs}^2}$.
Say $[\xi,\xi_1]$ is a jointly Gaussian random variable in some product space. Denote $\hat \xi_1 := \Ex{\xi|\xi_1}$ and $P_1:=\Cov{\hat\xi_1-\xi,\hat\xi_1-\xi}$. Then say
\begin{equation} \label{eq:out}
 \xi_2=H\xi+w 
 \end{equation}
 where $H$ is a bounded operator and $w$ is a Gaussian random variable with mean zero and covariance $R>0$ and $w$ is independent of $[\xi,\xi_1]$. Then it holds that 
\[
\hat\xi_2:=\Ex{\xi|\{\xi_1,\xi_2\}}=\hat\xi_1+P_1H^*(HP_1H^*+R)^{-1}(\xi_2-H\hat\xi_1).
\]
Then
\[
\hat\xi_2-\hat\xi_1=P_1H^*(HP_1H^*+R)^{-1}(H(\xi-\hat\xi_1)+w)
\]
from which it directly follows that
\[
\Cov{\hat\xi_2-\hat\xi_1,\hat\xi_2-\hat\xi_1}=P_1H^*(HP_1H^*+R)^{-1}HP_1,
\]
and further,
\begin{equation} \label{eq:trace}
\Ex{\bnorm{ \hat\xi_2-\hat\xi_1}^2}=\tr \left( P_1H^*(HP_1H^*+R)^{-1}HP_1\right).
\end{equation}
The basic ingredients for the proofs in this paper are now presented, namely the martingale $\tilde z_j$ defined as $\Ex{z(T)|y(\textup{T}_j)}$ with $\textup{T}_j$ defined in \eqref{eq:times}, the telescope identity \eqref{eq:telescope}, and equation \eqref{eq:trace} for the increment norm.

Later we sometimes need the assumption that $x \in \dom(A)$ almost
surely. With Gaussian random variables this means that $x$ is actually
a $\dom(A)$-valued random variable.
\begin{prop} \label{prop:X_1} Let $\xi$ be an $\sX$-valued
  Gaussian random variable s.t. $\xi \in \sX_1$ almost surely where
  $\sX_1 \subset \sX$ is another Hilbert space with continuous and
  dense embedding. Then $\xi$ is an $\sX_1$-valued Gaussian random
  variable.
\end{prop}
\begin{proof}
  Pick $h \in \sX_1$. We intend to show that $\ip{\xi,h}_{\sX_1}$ is a
  real-valued Gaussian random variable. For $h \in \sX_1$ there exists
  $h' \in \sX_1'$, the dual space of $\sX_1$,
  s.t. $\ip{\xi,h}_{\sX_1}=\ip{\xi,h'}_{(\sX_1,\sX_1')}$ and further,
  there exists a sequence $\{h_i\}_{i=1}^{\infty} \subset \sX$ such
  that $\ip{\xi,h'}_{(\sX_1,\sX_1')}=\lim_{i \to
    \infty}\ip{\xi,h_i}_{\sX}$. Now $\ip{\xi,h_i}_{\sX}$ is a pointwise
  converging sequence of Gaussian random variables and so the limit is
  also Gaussian.
\end{proof}
\noindent Fernique's theorem \cite[Theorem~2.6]{DaPrato} can be
applied to note that
if $\xi$ is an $\sX_1$-valued Gaussian random variable then $\xi \in
L^p(\Omega;\sX_1)$ for any $p > 0$. In particular,
$\Ex{\norm{\xi}_{\sX_1}^2}<\infty$ and if $A \in \mathcal{L}(\sX_1,\sX)$ then
$A\xi$ is an $\sX$-valued Gaussian random variable.

\section{Convergence results without input noise} \label{sec:noiseless}

Assume now that there is no input noise in the system \eqref{eq:system}, that is, $u=0$. The output is then given by 
\[
y(t)=C\int_0^t e^{As}x \, ds + w(t)
\]
where $w$ is a Brownian motion. In a sense, the output is parameterized by the initial state $x$, and therefore we also define the martingale $\tilde x_j := \Ex{x|y(\textup{T}_j}$. Then in the absence of input noise it holds that $\tilde z_j = e^{AT}\tilde x_j$.

In order to use \eqref{eq:trace} to compute one increment $\Ex{\norm{\tilde z_{j+1}-\tilde z_j}_{\Xs}^2}$,  we need to consider how to take into account an intermediary observation $y(t_{j+1})$ in the state estimate. Obviously the noise process value $w(t_{j+1})$ is not independent of the  measurements $y(t_i)$ with $i=1,...,j$ as required in order to use \eqref{eq:trace}. 
As in \cite{tempI}, the dependence of the noise term $w(t_{j+1})$ on $y(t_i)$ with $i=1,...,j$ is removed by subtracting from $y(t_{j+1})$ the linear interpolant $\frac12 \big( y(t_{j+1}-h)+y(t_{j+1}+h) \big)$ where $h=\frac{T}{2^Kn}$ with the corresponding $K$ (see \eqref{eq:times} and Figure~\ref{fig:times}). Note that $t_{j+1} \pm h \in \textup{T}_j$. When there is no input noise, this new output is obtained as
\begin{equation} \label{eq:tildey}
\tilde y_{j+1}=C_h(t_{j+1})x+\tilde w_{j+1}
\end{equation}
where $C_h(t)$ is defined for $x \in \Xs$ and $t \ge h$ by
\begin{equation} \label{eq:C_h}
C_h(t)x:=\frac{C}2 \left( \int_{t-h}^t e^{As}x \, ds -\int_t^{t+h} e^{As}x \, ds \right),
\end{equation}
and $\tilde w_{j+1} \sim N \left(0, \frac{h}2 R \right)$ is independent of $y(t_i)$ and hence of $\tilde y_i$ for $i=1,...,j$. Note that we can first define $C_h(t)x$ only for $x \in \dom(A)$. However, it holds that
  \begin{align*} \norm{\int_t^{t+h}e^{As}x \, ds}_{\dom(A)}^2 &
    =\norm{\int_t^{t+h}e^{As}x \,
      ds}_{\sX}^2+\norm{A\int_t^{t+h}e^{As}x \, ds}_{\sX}^2 \\
    &\le h^2\mu^2\norm{x}_{\sX}^2+\norm{\left(e^{A(t+h)}-e^{At}\right)x}_{\sX}^2
    \le \left(h^2+4 \right)\mu^2 \norm{x}_{\sX}^2.
\end{align*}
Therefore, we can uniquely extend $C_h(t)$ to a continuous operator from $\Xs$ to $\Ys$, and $\norm{C_h(t)}_{\mathcal{L}(\sX,\sY)} \le \mu\sqrt{h^2+4}
\norm{C}_{\mathcal{L}(\dom(A),\sY)}$. In addition, using $Ce^{As}\xi=Ce^{At}\xi+\int_t^s CAe^{Ar}\xi dr$ in \eqref{eq:C_h} yields a useful bound
\begin{equation} \label{eq:C_hder}
\norm{C_h(t)\xi}_{\sY} \le \frac12 \! \left( \int_{t-h}^t  \int_s^t \! \norm{CAe^{Ar}\xi}_{\sY}dr \, ds+ \!\int_t^{t+h} \!\! \int_t^s \! \norm{CAe^{Ar}\xi}_{\sY}dr \, ds \!  \right) \!\!
\end{equation}
provided that $CAe^{Ar}\xi \in L^1(t-h,t+h;\sY)$.

Now the form \eqref{eq:tildey} is exactly as \eqref{eq:out} and so (recalling $\tilde z_j=e^{AT}\tilde x_j$) we can use \eqref{eq:trace} to obtain
\begin{align} \label{eq:incnorm}
\Ex{\norm{\tilde z_{j+1}-\tilde z_j}_{\Xs}^2} & = \tr \Bigg( e^{AT} P_jC_h(t_{j+1})^*\times  \\ &  \times \!\left(  \! C_h(t_{j+1})P_jC_h(t_{j+1})^* \! +\frac{h}2 R \right)^{-1}\!\! C_h(t_{j+1})P_j e^{A^*T} \!  \Bigg). \nonumber
\end{align}
Finally we are able to show the main lemma, which links the convergence of the state estimate to properties of the operator $C_h(t)$ and the smoothness of the output $y$. This lemma serves as the basis for all the proofs of our main theorems when there is no input noise.
\begin{lma} \label{lma:batch}
Let $\sX_1$ be a dense Hilbert subspace of $\sX$ with a continuous embedding. Assume that there exist $M > 0$ and $k>1$ such that for any $K \in \mathbb{N}$ and $\xi \in \sX_1$, it holds that
\begin{equation} \label{eq:batch}
\sum_{j=2^{K-1}n+1}^{2^Kn} \!\!\! {\norm{C_h(t_j)\xi}_{\Ys}^2} \le M h^k \norm{\xi}_{\sX_1}^2
\end{equation}
where $h=\frac{T}{2^Kn}$. Assume also that $x \in \sX_1$ almost surely and $u=0$. Then for $\hat z_{T,n}$ and $\hat z(T)$ defined in \eqref{eq:estimates}, it holds that
\[
\Ex{\norm{\hat{z}_{T,n}-\hat{z}(T)}_{\Xs}^2} \le \frac{2M  \Ex{\norm{x}_{\sX_1}^2} \Ex{\norm{\hat z_{T,n}-z(T)}_{\Xs}^2}}{(2^{k-1}-1)\min(\eig(R))}\frac{T^{k-1}}{n^{k-1}}.
\]
\end{lma}
\noindent Note that $\Ex{\norm{x}_{\sX_1}^2}$ is well defined and finite by Proposition~\ref{prop:X_1}. A strict \emph{a priori} result is obtained by replacing $\Ex{\norm{\hat z_{T,n}-z(T)}_{\Xs}^2} \le \mu^2 \Ex{\norm{x}_{\sX}^2}$.
\begin{proof}
We will use a shorter notation $C_h=C_h(t_{j+1})$ in the proof. Note that
\begin{equation} \nonumber 
\norm{\left( C_hP_jC_h^*+\frac{h}2 R \right)^{-1}} \le \frac{2}{h \min(\eig(R))}=:\frac{C_R}h.
\end{equation}
One increment can be bounded by \eqref{eq:incnorm}:
\begin{align*} 
  & \tr \! \left( \! e^{AT}P_j C_h^* \left( C_h P_j
      C_h^*+\frac{h}{2}R \right)^{\! -1} \!\! C_h P_j e^{A^*T}\right) \\& =\sum_{k=1}^{\infty} 
  \ip{C_hP_je^{A^*T}e_k,\left( C_h P_j C_h^*+\frac{h}{2}R \right)^{\! -1} \!\! C_h P_je^{A^*T}e_k}_{\! \sY} \\
   & \le \frac{C_R}{h}\sum_{k=1}^{\infty} \norm{C_hP_je^{A^*T}e_k}_{\sY}^2=\frac{C_R}{h}\sum_{k=1}^{\infty} \norm{\Ex{C_h(\tilde x_j-x)\ip{e^{AT}(\tilde x_j-x),e_k}_{\sX}}}_{\sY}^2 
\\
  & \le \frac{C_R}{h}\Ex{\norm{C_h(\tilde
      x_j-x)}_{\sY}^2} \sum_{k=1}^{\infty} \Ex{\ip{e^{AT}(\tilde x_j-x),e_k}_{\sX}^2} \\        &\le \frac{C_R}{h} \tr(C_hP_jC_h^*)\Ex{\norm{\tilde z_j-z(T)}_{\sX}^2}.
\end{align*}
Now we use this bound, $P_j \le P_n \le P_0$, and assumption \eqref{eq:batch} to sum up the increments corresponding to $h=\frac{T}{2^K n}$, that is, $j=2^{K-1}n,...,2^Kn-1$:
\begin{align*}
& \sum_{j=2^{K-1}n}^{2^Kn-1} \Ex{\norm{\tilde z_{j+1}-\tilde z_j}_{\Xs}^2} \le M C_R  \Ex{\norm{x}_{\sX_1}^2} \Ex{\norm{\tilde z_n-z(T)}_{\Xs}^2} \left(\frac{T}{2^{K}n}\right)^{k-1}
\end{align*}
and finally, summing up over $K=1,2,...$ gives the result.
\end{proof}

\subsection{Diagonalizable main operator}

We proceed to prove a convergence result for systems with unbounded
observation operator $C$ --- provided that $A$ is (unitarily)
diagonalizable.  The proof is based on Lemma~\ref{lma:batch}. 
 To get a useful bound for $\norm{C_h\xi}_{\Ys}^2$, some assumptions on the degree of unboundedness of $C$ and the spectral asymptotics
of $A$ are required.
\begin{thm} \label{thm:unbounded} Let $\hat{z}_{T,n}$ and
  $\hat{z}(T)$ be as defined above in \eqref{eq:estimates}. Denote by
  $\left\{\lambda_k \right\}_{k=1}^{\infty} \subset \mathbb{C}$ the spectrum of
  $A$ ordered so that $|\lambda_k|$ is non-decreasing and let
  $\left\{ e_k \right\}_{k=1}^{\infty} \subset \dom(A)$ be the
  corresponding set of eigenvectors that give an orthonormal basis for
  $\sX$. Make the following assumptions on $x$, $A$, and $C$:
\begin{itemize}
\item[(i)] $x \in \dom(A)$ almost surely;
\item[(ii)] There exists $\delta >
  1/2$ such that
  \begin{equation} \nonumber \lim_{k \to
      \infty}\frac{|\lambda_k|}{k^{\beta}}= \left\{ \begin{array}{ll} 0 & \textrm{when } \beta > \delta, \\
\infty &
      \textrm{when } \beta < \delta; \end{array} \right.
\end{equation}
\item[(iii)] There exists $\gamma \in [0,1)$ such that $2\gamma+1/\delta<2$ and
\begin{equation} \nonumber
\sup_k \frac{\norm{Ce_k}_{\sY}}{|\lambda_k|^{\gamma}}<\infty.
\end{equation}
\end{itemize}
Then the following holds:
\begin{itemize}
\item If $\lim_{k \to
    \infty}\frac{|\lambda_k|}{k^{\delta}}=\Gamma \in
  (0,\infty)$, then
\begin{equation} \nonumber
 \Ex{\norm{\hat{z}_{T,n}-\hat{z}(T) }_{\sX}^2} \le
  \frac{MT^{3-2\gamma-1/\delta}}{n^{2-2\gamma-1/\delta}}
\end{equation}
where the constant $M$ is given below in \eqref{eq:constant_M}.

\item If either this limit does not exist, or it is 0 or $\infty$, then
  for all $\epsilon \in \left(0, \delta-\frac1{2-\gamma}
  \right)$ 
  \begin{equation} \nonumber \Ex{\norm{\hat{z}_{T,n}-\hat{z}(T)
      }_{\sX}^2} \le
    \frac{M_{\epsilon}T^{3-2\gamma-1/(\delta+\epsilon)}}{n^{2-2\gamma-1/(\delta-\epsilon)}}
\end{equation}
where the $\epsilon$-dependent constant $M_{\epsilon}$ is given below
also in \eqref{eq:constant_M} but with different, $\epsilon$-dependent
parameters (see the last paragraph of the proof).

\end{itemize}

\end{thm}
\noindent For example, 1D wave equation on interval $[0,L]$
with Dirichlet boundary conditions in the natural state space where
some pointwise value of the state is observed, satisfies the
assumptions of the above theorem with $\delta=1$ and $\gamma=0$. The
limit of $\frac{|\lambda_k|}{k}$ as $k \to \infty$ exists and
it is $\Gamma=\frac{\pi}{2L}$. This would imply convergence rate
$\Ex{\norm{\hat{z}_{T,n}-\hat{z}(T) }_{\sX}^2} \le \frac{MT^2}{n}$.

\begin{proof}
Assume first that $\lim_{k \to \infty}\frac{|\lambda_k|}{k^{\delta}}=\Gamma \in
  (0,\infty)$. Denote $\xi =\sumk \alpha_k e_k \in \dom(A)$ which is equivalent to $\sumk |\lambda_k|^2 \alpha_k^2 \le \infty$. Now 
  \begin{equation} \label{eq:C_hx} C_h\xi=\sumk \frac{\alpha_k}{2}
  \left( \int_{t-h}^t e^{\lambda_ks}ds- \! \int_t^{t+h} \!
    e^{\lambda_ks}ds \right)Ce_k.
\end{equation}
For the term inside parentheses, we have 
\begin{equation} \nonumber \left| \int_{t-h}^t
    e^{\lambda_ks}ds-\int_t^{t+h}
    e^{\lambda_ks}ds\right| \le h^2\sup_{s \ge 0}
  \left|\frac{d}{ds}e^{\lambda_ks} \right| \le
 \mu h^2|\lambda_k|,
\end{equation}
since $\norm{e^{At}}_{\mathcal{L}(\sX)} \le \mu$.
On the other hand, computing the integrals yields
\begin{equation} \nonumber \left| \int_{t-h}^t
    e^{\lambda_ks}ds-\int_t^{t+h}
    e^{\lambda_ks}ds\right| \le \frac{4\mu}{|\lambda_k|}.
\end{equation}
Now the idea is to bound the sum in \eqref{eq:C_hx} by using the first
bound for small $k$ and the latter for large $k$. Define the index
$n(h):=\lceil h^{-1/\delta} \rceil$ for splitting the sum to get
\begin{align*} \norm{C_h\xi}_{\sY} &\le
  \sum_{k=1}^{n(h)}\frac{|\alpha_k|}{2}\mu\norm{Ce_k}_{\sY}|\lambda_k|h^2
  \, + \!\!\!\!\! \sum_{k=n(h)+1}^{\infty}\!\!\!\!\!
  |\alpha_k|\norm{Ce_k}_{\sY}\frac{2\mu}{|\lambda_k|}=:(I)+(II). 
\end{align*}
We then proceed to find upper bounds for the two parts. Using
Cauchy-Schwartz inequality and denoting $\hat{\Gamma}:=\sup_k
\frac{|\lambda_k|}{k^{\delta}}$ gives
\begin{align*} (I) &\le \mu\frac{h^2}2
  \left(\sum_{k=1}^{n(h)}\alpha_k^2\norm{Ce_k}_{\sY}^2|\lambda_k|^{2-2\gamma}
  \right)^{\!\! 1/2} \!   \left(\sum_{k=1}^{n(h)}|\lambda_k|^{2\gamma} \right)^{\!\! 1/2}\le \frac{\mu h^2\hat{\Gamma}^{\gamma}}2 
    M_I \left( \sum_{k=1}^{n(h)} k^{2\gamma\delta}\right)^{\!\! 1/2} 
\end{align*}
where
$M_I=\left(\sum_{k=1}^{n(h)}\alpha_k^2\norm{Ce_k}_{\sY}^2|\lambda_k|^{2-2\gamma}
\right)^{1/2}$. The sum inside the parentheses can be bounded from
above by the integral $\int_0^{n(h)+1} x^{2 \gamma \delta} dx$ to get
\begin{align*}
(I)&
\le  \frac{\mu h^2\hat{\Gamma}^{\gamma}}{2\sqrt{2\gamma\delta+1}}M_I \sqrt{(n(h)+1)^{2\gamma\delta+1}} \le
  \frac{3^{\delta}\mu\hat{\Gamma}^{\gamma}}{2\sqrt{2\gamma\delta+1}}M_Ih^{2-\gamma-\frac1{2\delta}}
  \le \frac{3^{\delta}\mu\hat{\Gamma}^{\gamma}}2 M_I
  h^{2-\gamma-\frac1{2\delta}}
\end{align*}
where the last row follows from the facts that
\begin{equation} \nonumber \sqrt{(n(h)+1)^{2\gamma\delta+1}}\le
  \sqrt{(h^{-1/\delta}+2)^{2\gamma\delta+1}}=(1+2h^{1/\delta})^{\gamma\delta+\frac12}h^{-\gamma-\frac1{2\delta}} \le 3^{\delta}h^{-\gamma-\frac1{2\delta}}
\end{equation}
if $h \le 1$, and that $2\gamma\delta+1>1$.

For the second part, assume $|\lambda_k| \ge
\check{\Gamma}k^{\delta}$ for $k \ge n(h)+1$ where
$\check{\Gamma}=0.9\Gamma$ for example. Again, using Cauchy-Schwartz
inequality yields
\begin{align*} 
(II) &\le
  2\mu\left(\sum_{k=n(h)+1}^{\infty}\alpha_k^2\norm{Ce_k}_{\sY}^2|\lambda_k|^{2-2\gamma}\right)^{\!\! 1/2}\left(\sum_{k=n(h)+1}^{\infty}\frac1{|\lambda_k|^{4-2\gamma}}\right)^{\!\! 1/2} \\
  & \le \frac{2\mu}{\check{\Gamma}^{2-\gamma}}
  M_{II}\left(\sum_{k=n(h)+1}^{\infty}\frac1{k^{(4-2\gamma)\delta}}\right)^{\!\! 1/2} 
\end{align*}
where
$M_{II}=\left(\sum_{k=n(h)+1}^{\infty}\alpha_k^2\norm{Ce_k}_{\sY}^2|\lambda_k|^{2-2\gamma}\right)^{\!\!
  1/2}$. Now the sum inside the parentheses can be bounded from above
by the integral $\int_{n(h)}^{\infty}\frac1{x^{(4-2\gamma)\delta}}
dx$. Note that our assumptions on $\gamma$ and $\delta$ imply
$(4-2\gamma)\delta>2$. So we get
\begin{equation} \nonumber
  (II) \le \frac{2 \mu M_{II}}{\check{\Gamma}^{2-\gamma}\sqrt{(4-2\gamma)\delta-1}}
  \left( \frac1{n(h)^{(4-2\gamma)\delta-1}} \right)^{\!\! 1/2} 
   \le
  \frac{2\mu}{\check{\Gamma}^{2-\gamma}}M_{II}h^{2-\gamma-\frac1{2\delta}} 
\end{equation}
where in the last row we have used $n(h) \ge h^{-1/\delta}$.

Combining the bounds gives
\begin{align*}
 {\norm{C_h\xi}_{\sY}^2} & \le 2 \big((I)^2+(II)^2\big) \\ &\le 2 \mu^2 \left( M_I^2+M_{II}^2 \right)\max \left( \frac{9^{\delta}\hat{\Gamma}^{2\gamma}}4 , 
  \frac{4}{\check{\Gamma}^{4-2\gamma}}
\right)h^{4-2\gamma-1/\delta} \\
&\le
2\mu^2{\norm{A\xi}_{\sX}^2}\sup_k\frac{\norm{Ce_k}_{\sY}^2}{|\lambda_k|^{2\gamma}}
\max \left( \frac{9^{\delta}\hat{\Gamma}^{2\gamma}}4 , 
  \frac{4}{\check{\Gamma}^{4-2\gamma}}
\right)h^{4-2\gamma-1/\delta}
\end{align*}
where we have used
\begin{equation} \nonumber
M_I^2+M_{II}^2=\sum_{k=1}^{\infty}\alpha_k^2\norm{Ce_k}_{\sY}^2|\lambda_k|^{2-2\gamma} \le \sum_{k=1}^{\infty}|\lambda_k|^2\alpha_k^2\sup_j\frac{\norm{Ce_j}_{\sY}^2}{|\lambda_j|^{2\gamma}}.
\end{equation}

Note that we assumed that we could choose for example
$\check{\Gamma}=0.9\Gamma$. In some sense this is not our choice but
we need to make sure that the ``original'' $h=\frac{T}{2n}$ is small
enough so that $n(T/(2n))=\left( \frac{2n}{T} \right)^{1/\delta}$ is
such that there exists $\check{\Gamma}>0$ for which
$\frac{|\lambda_k|}{k^{\delta}} \ge \check{\Gamma}$ for $k \ge
n(T/(2n))$.

To get a bound for the sum in \eqref{eq:batch}, we simply multiply the bound obtained for $\norm{C_h\xi}_{\sY}^2$ by $2^{K-1}n=\frac{T}{2h}$ to get
\[
\sum_{j=2^{K-1}n+1}^{2^Kn} \!\!\! {\norm{C_h(t_j)\xi}_{\Ys}^2} \le \sup_k\frac{\norm{Ce_k}_{\sY}^2}{|\lambda_k|^{2\gamma}}
\max \left( \frac{9^{\delta}\hat{\Gamma}^{2\gamma}}4 , 
  \frac{4}{\check{\Gamma}^{4-2\gamma}}
\right)\mu^2T{\norm{\xi}_{\dom(A)}^2}h^{3-2\gamma-1/\delta}
\]
and so the result follows by Lemma~\ref{lma:batch} with
\begin{equation} \label{eq:constant_M} M=
  \frac{2\mu\Ex{\norm{\hat z_{T,n}-z(T)}_{\Xs}^2}\Ex{\norm{x}_{\dom(A)}^2}}{(2^{2-2\gamma-1/\delta} -1)\min(\eig(R))}\sup_k\frac{\norm{Ce_k}_{\sY}^2}{|\lambda_k|^{2\gamma}}\max
  \!  \left( \! \frac{9^{\delta}\hat{\Gamma}^{2\gamma}}4,
    \frac{4}{\check{\Gamma}^{4-2\gamma}} \! \right).
\end{equation}

In the case that $\lim_{k \to
  \infty}\frac{|\lambda_k|}{k^{\delta}}$ is 0, $\infty$, or it
does not exist, some modifications are required to the bounds of $(I)$
and $(II)$. In the bound for $(I)$, $\delta$ needs to be replaced by
$\delta+\epsilon$ and then $\hat{\Gamma}_{\epsilon}=\sup_k
\frac{|\lambda_k|}{k^{\delta+\epsilon}}<\infty$. In the bound
for $(II)$, $\delta$ needs to be replaced by $\delta-\epsilon$ and
then $\check{\Gamma}_{\epsilon}=\inf_{k \ge n(h)+1}
\frac{|\lambda_k|}{k^{\delta-\epsilon}}>0$.
\end{proof}
The assumption {\it (iii)} in the theorem differs from our minimal
assumption $C \in \mathcal{L}(\dom(A),\sY)$ which is equivalent to
$\left\{ \frac{\norm{Ce_k}_{\sY}}{|\lambda_k|} \right\} \in
l^2$ for unitarily diagonalizable $A$. It is possible to construct a
system for which $C \in \mathcal{L}(\dom(A),\sY)$ but {\it (iii)} does
not hold.

\begin{rmk} {\rm
  Theorem \ref{thm:unbounded} can be extended to $\gamma < 0$. In that
  case, when determining the bounds for $(I)$ and $(II)$, the
  computations are carried out as if $\gamma$ were zero.
This eventually leads to a bound $
\Ex{\norm{\hat{z}_{T,n}-\hat{z}(T) }_{\sX}^2} \le
\frac{MT^{3-1/\delta}}{n^{2-1/\delta}}$. Note that if assumption {\it
  (iii)} holds for $\gamma < -\frac1{2\delta}$ then $C$ is actually
bounded.}
\end{rmk}

\subsection{Admissible observation operator}

 In the next result we assume that the observation operator $C$ is admissible in the sense of Weiss \cite{Weiss:admissible}. One good example of systems that satisfy assumption (iii) in the following theorem
is provided by scattering passive boundary control systems, see the
article \cite{M-S:BNDR} by Malinen and Staffans. For a more extensive background, we refer to \cite{WPLS} by Staffans.

\begin{thm} \label{thm:unbounded2} Let $\hat{z}_{T,n}$ and
  $\hat{z}(T)$ be as defined above in \eqref{eq:estimates} and $u=0$.
Make the following assumptions:
\begin{itemize}
\item[(i)] $x \in \dom(A)$ almost surely; 
\item[(ii)] The orthonormal basis $\{ e_k \} \subset \sX$ is such that
  $e_k \in \dom(A^2)$ for every $k \in \mathbb{N}$ and there exists
  $\delta > 1/2$ such that for $\xi=\sumk \alpha_k e_k$ the norm given
  by $\sqrt{\sumk k^{2\delta} \alpha_k^2}$ is equivalent to the
  $\dom(A)$-norm and $\sqrt{\sumk k^{4\delta} \alpha_k^2}$ is
  equivalent to the $\dom(A^2)$-norm;
\item[(iii)] The observation operator is admissible, that is, for any $T \ge 0$ there exists $H_T \ge 0$,  such that
  $\norm{Ce^{A(\cdot)}x}_{L^2((0,T);\sY)} \le H_T \norm{x}_{\sX}$.
\end{itemize}
 Then
\begin{equation} \nonumber
\Ex{\norm{\hat{z}_{T,n}-\hat{z}(T)}_{\sX}^2}
    \le \frac{M(T) \, T^{2-1/2\delta}}{n^{1-1/2\delta}}
\end{equation}
with $M(T)=\frac{2 \Ex{\norm{\hat z_{T,n}-z(T)}_{\Xs}^2}\Ex{\norm{x}_{\dom(A)}^2}}{(2^{1-1/2\delta}-1)\min(\eig(R))} \max
\left( \frac{3^{2\delta+1}T\mu
    \norm{C}_{\mathcal{L}(\dom(A),\sY)}^2}{8\delta+4},\frac{H_T}{2\delta-1}
\right)$.
\end{thm}
\begin{proof}
  In this proof, the aforementioned norms are used in $\dom(A)$ and
  $\dom(A^2)$. We need to utilize the global output bound
  $\norm{Ce^{A(\cdot)}x}_{L^2((0,T);\sY)} \le H_T \norm{x}_{\sX}$. To this end,
  define a stacked operator $\widehat
  C_h:=[C_h(h),C_h(3h),\dots,C_h(T-h)]^T$ for $h=\frac{T}{2^Kn}$
  mapping to a product space $\sY^{2^{K-1}n}$. Then the sum on the left hand side of \eqref{eq:batch} is obtained as $\bnorm{\widehat C_h \xi}_{\sY^{2^{K-1}n}}^2$.
    In this proof, $\big[ a_i \big]_{i=1}^N$ is used to
  denote an augmented vector with $N$ components $a_i$. 

The
  proof proceeds similarly as the proof of Theorem~\ref{thm:unbounded}
  but the sum in \eqref{eq:C_hx} is split using the index $n(h)=\lceil
  h^{-1/(2\delta)}\rceil$ to get $\bnorm{\widehat C_h\xi}_{\sY^{2^{K-1}n}} \le (I)+(II)$ where
\begin{align*}
   (I) &=
\norm{\sum_{k=1}^{n(h)} \frac{\alpha_k}2  \left[
      C  \int_{(2j-2)h}^{(2j-1)h} \!\! e^{As}e_k \, ds -
      C  \int_{(2j-1)h}^{2jh} \! e^{As}e_k \, ds
    \right]_{i=1}^{2^{K-1}n}}_{\sY^{2^{K-1}n}} \\
  &\le \sum_{k=1}^{n(h)}\frac{|\alpha_k|}2 \sqrt{\frac{T}{2h}}h^2
  \mu\norm{C}_{\mathcal{L}(\dom(A),\sY)}k^{2\delta} \\ & \le
  \sqrt{\frac{Th^3}{8}}\mu\norm{C}_{\mathcal{L}(\dom(A),\sY)}
  \left(\sum_{k=1}^{n(h)}k^{2\delta}\alpha_k^2\right)^{\!\! 1/2} \!\! \left(\sum_{k=1}^{n(h)}k^{2\delta}\right)^{\!\! 1/2}
\end{align*}
where the first inequality is obtained using \eqref{eq:C_hder} with $t=(2j-1)h$ and bounding the derivative of $Ce^{Ar}e_k$ by
\begin{equation} \nonumber \norm{ CAe^{Ar}e_k}_{\sY} \le
    \mu\norm{C}_{\mathcal{L}(\dom(A),\sY)} \norm{e_k}_{\dom(A^2)} =
    \mu\norm{C}_{\mathcal{L}(\dom(A),\sY)}k^{2\delta}
\end{equation}
 and noting that $2^{K-1}n=\frac{T}{2h}$. 

 For the remaining part it holds that
\begin{align*}
  (II)&=\norm{\sum_{k=n(h)+1}^{\infty} \! \frac{\alpha_k}2 \left[
      C\int_{(2i-2)h}^{(2i-1)h} \! e^{As}e_k \, ds -
      C\int_{(2i-1)h}^{2ih} \! e^{As}e_k \, ds
    \right]_{i=1}^{2^{K-1}n}}_{\sY^{2^{K-1}n}} \\
  &\le \!\! \sum_{k=n(h)+1}^{\infty} \! \frac{|\alpha_k|}2\sqrt{2h} H_T \le
  \sqrt{\frac{h}2} H_T
  \left(\sum_{k=n(h)+1}^{\infty} \!\! k^{2\delta}\alpha_k^2\right)^{\!\! 1/2} \!\! \left(\sum_{k=n(h)+1}^{\infty}\!\! k^{-2\delta}\right)^{\!\! 1/2}
\end{align*}
since it holds that
\begin{align*}
  &\norm{\left[ C\int_{(2i-2)h}^{(2i-1)h}e^{As}e_k \, ds -
      C\int_{(2i-1)h}^{2ih}e^{As}e_k \, ds
    \right]_{i=1}^{2^{K-1}n}}_{\sY^{2^{K-1}n}} \\
  & \le \left(\sum_{i=1}^{2^{K-1}n}
    \left(\int_{(2i-2)h}^{2ih}\norm{Ce^{As}e_k}_{\sY} ds \right)^{\!\! 2 \,}
  \right)^{\!\! 1/2} \le \sqrt{2h} \norm{Ce^{As}e_k}_{L^2((0,T);\sY)}
\end{align*}
where the last inequality follows from Cauchy-Schwartz inequality. Finally, the sums $\sum_{k=1}^{n(h)}k^{2\delta}$ and $\sum_{k=n(h)+1}^{\infty}\! k^{-2\delta}$ are bounded by integrals of $x^{2\delta}$ and $x^{-2\delta}$, respectively, as in the proof of Theorem~\ref{thm:unbounded} and
then the result is obtained by Lemma~\ref{lma:batch} and $\bnorm{\widehat C_h\xi}_{\sY^{2^{K-1}n}}^2 \le 2\big((I)^2+(II)^2\big)$.
\end{proof}

\subsection{Analytic semigroup}

In this section we show the convergence estimate when $A$ is the
generator of an analytic semigroup. One result is first shown without
additional assumptions for bounded and unbounded observation operator
$C$. Then we assume further that $-A$ is a sectorial operator
in $\sX$ which enables us to treat non-integer
powers $(-A)^{\eta}$ for $\eta \ge 0$. An example of such case is
provided by heat equation treated below in Example~\ref{ex:heat}.

An important tool here is that for analytic semigroups it holds
that
\begin{equation} \label{eq:an} \norm{A^{\kappa}e^{At}}_{\mathcal{L}(\sX)}\le
  \frac{c(\kappa)}{t^{\kappa}}, \qquad t > 0, \ \kappa \in \mathbb{N}
\end{equation}
(see \cite[Theorem 3.3.1]{tanabe}). Using this to bound the derivative of $Ce^{At}\xi$, that is $CAe^{At}\xi$, gives by \eqref{eq:C_hder},
\begin{equation} \label{eq:C_han} \norm{C_h(t)\xi}_{\sY} \le \frac{c(1)\norm{C}_{\mathcal{L}(\sX/\dom(A),\sY)}
    \norm{\xi}_{\sX/\dom(A)}}{2(t-h)}
h^2, \qquad t>h.
\end{equation}
\begin{thm} \label{thm:analytic1} Let $\hat{z}_{T,n}$ and
  $\hat{z}(T)$ be as defined above in \eqref{eq:estimates}. Assume $A$
  is the generator of an analytic  $C_0$-semigroup and
  assume either
\begin{itemize}
\item[(i)] $C \in \mathcal{L}(\sX,\sY)$, or
\item[(ii)] $C \in \mathcal{L}(\dom(A),\sY)$ and $x \in \dom(A)$ almost surely.
\end{itemize}
Then
\begin{equation} \nonumber
\Ex{\norm{\hat{z}_{T,n}-\hat{z}(T)}_{\sX}^2} \le \frac{MT}{n}
\end{equation}
where  $M=\frac{2\norm{C}_{\mathcal{L}(\sX/\dom(A),\sY)}^2}{\min(\eig(R))} \left(\mu^2+\frac{c(1)^2\pi^2}{96} \right)
\Ex{\norm{z(T)-\hat z_{T,n}}_{\sX}^2}\Ex{\norm{x}_{\sX/\dom(A)}^2}
$.
\end{thm}
\begin{proof}
  The proofs for the two cases are identical so only the case 
    (i) is presented. In the second case just replace $\sX$ by $\dom(A)$ in
  $\norm{\xi}_{\sX}$ and $\norm{C}_{\mathcal{L}(\sX,\sY)}$.

The proof is based on Lemma~\ref{lma:batch}, so we need to find a bound for
 \begin{equation}  \label{eq:sumbd}
 \sum_{j=2^{K-1}n+1}^{2^Kn} {\norm{C_h(t_j)\xi}_{\sY}^2}=\sum_{l=1}^{2^{K-1}n} \norm{C_h\big((2l-1)h\big)\xi}_{\sY}^2
 \end{equation}
 For $l=1$, we use
$\norm{C_h(h)\xi}_{\sY} \le h\mu
\norm{C}_{\mathcal{L}(\sX,\sY)}\norm{\xi}_{\sX}$ which is clear from the definition of $C_h(t)$ in \eqref{eq:C_h}. For $l > 1$, we use \eqref{eq:C_han} where the denominator becomes $2\big((2l-1)h-h\big)=4h(l-1)$ and so,
\begin{align*} 
 \sum_{j=2^{K-1}n+1}^{2^Kn} \!\! {\norm{C_h(t_j)\xi}_{\sY}^2}
    & \le \norm{C}_{\mathcal{L}(\sX,\sY)}^2
   {\norm{\xi}_{\sX}^2}
    \left(\mu^2+\frac{c(1)^2}{16} \!\! \sum_{l=2}^{ \ \ 2^{K-1}n } \!\! \frac1{(l-1)^2} \right)h^2  \\
& \le
\norm{C}_{\mathcal{L}(\sX,\sY)}^2 \norm{\xi}_{\sX}^2
  \left(\mu^2+\frac{c(1)^2\pi^2}{96} \right)h^2.
\end{align*}
The result follows by Lemma~\ref{lma:batch}.
\end{proof}

One more case is treated where $A$ is as before and, in addition,
$-A$ is a sectorial operator, see \cite[Section~3.8]{Arendt} for
definitions.
Then it is possible to define non-integer powers $(-A)^{\eta}$ where
$\eta \in \mathbb{R}$ and spaces $\dom ((-A)^{\eta})$ equipped with
the corresponding graph norm. Also \eqref{eq:an} holds then for
non-integer $\kappa \ge 0$ if $A$ is replaced by $-A$, see
\cite[Thm.~3.3.3]{tanabe}. In particular, if $A$ is strictly negative
definite, then it is sectorial. This type of systems are also studied
in \cite{DaPrato_Ichikawa} and \cite{Flandoli86}.

\begin{thm} \label{thm:analytic2} Let $\hat{z}_{T,n}$ and
  $\hat{z}(T)$ be as defined above in \eqref{eq:estimates}. Assume that
  $A$ is the generator of an analytic $C_0$-semigroup
  and, in addition, $-A$ is a sectorial operator. Then assume $C \in
  \mathcal{L}(\dom((-A)^{\nu}),\sY)$ and $x \in \dom((-A)^{\eta})$
  almost surely where $\nu \in \mathbb{R}$ and $\eta\in \mathbb{R}$
  are such that $|\eta-\nu|<1/2$. Then\footnote{This result extends to
    $\eta-\nu=1/2$ in which case the convergence rate is
    $\mathcal{O}(T^2n^{-2}\ln n)$.}
\begin{equation} \nonumber \Ex{\norm{\hat{z}_{T,n}-\hat{z}(T)}_{\sX}^2}
  \le \frac{MT^{1+2(\eta-\nu)}}{n^{1+2(\eta-\nu)}}
\end{equation}
where $M$ is given below in \eqref{eq:Man}.
\end{thm}

\begin{proof}
  This is done exactly as the proof of Theorem~\ref{thm:analytic1}
  above. Just the bounds for $\norm{C_h((2l-1)h)\xi}_{\sY}$ in
  the summation \eqref{eq:sumbd} are computed differently. To begin
  with, we note that \eqref{eq:an} with non-integer
  $\kappa=1-\eta+\nu$ yields,
  \begin{equation} \label{eq:C_hni} \norm{CAe^{At}\xi}_{\sY}
    \le \norm{C}_{\mathcal{L}(\dom((-A)^{\nu}),\sY)}\norm{\xi}_{\dom((-A)^{\eta})} \!
    \frac{c(1-\eta+\nu)}{t^{1-\eta+\nu}}.
\end{equation}

When treating the term with $l=1$ in \eqref{eq:sumbd}, the cases $\nu
\ge \eta$ and $\nu < \eta$ have to be considered separately. First for
$\nu \ge \eta$, 
  \begin{align*}
    \norm{Ce^{At}\xi}_{\sY} &=\norm{C(-A)^{-\nu}(-A)^{\nu-\eta}e^{At}(-A)^{\eta}\xi}_{\sY} \\
    &\le \norm{C}_{\mathcal{L}(\dom((-A)^{\nu}),\sY)}\norm{\xi}_{\dom((-A)^{\eta})}\frac{c(\nu-\eta)}{t^{\nu-\eta}}.
\end{align*}
Then directly by the definition of $C_h$ in \eqref{eq:C_h} (recalling $1+\eta-\nu >0$),
\begin{align*} 
  \norm{C_h(h)\xi}_{\sY}
  &\le\norm{C}_{\mathcal{L}(\dom((-A)^{\nu}),\sY)}\! \norm{\xi}_{\dom((-A)^{\eta})}c(\nu-\eta)
  \int_0^{2h} \! \frac1{s^{\nu-\eta}} ds \\
  &\le\norm{C}_{\mathcal{L}(\dom((-A)^{\nu}),\sY)}\! \norm{\xi}_{\dom((-A)^{\eta})}\frac{c(\nu-\eta)}{1+\eta-\nu}(2h)^{1+\eta-\nu}.
\end{align*}
For $\nu <\eta < 1+\nu$, use \eqref{eq:C_hder} with $t=h$. Using \eqref{eq:C_hni} to bound the derivative norm $\norm{CAe^{At}\xi}_{\sY}$ and computing the integrals yields
\begin{align*}  
  \norm{C_h(h)\xi}_{\sY}  & \le
  \norm{C}_{\mathcal{L}(\dom((-A)^{\nu}),\sY)}\norm{\xi}_{\dom((-A)^{\eta})}
  \frac{c(1-\eta+\nu)}{2(\eta-\nu)}\frac{2^{1+\eta-\nu}-2}{1+\eta-\nu}h^{1+\eta-\nu} \\
  &\le \norm{C}_{\mathcal{L}(\dom((-A)^{\nu}),\sY)}\norm{\xi}_{\dom((-A)^{\eta})} \frac{2 \ln 2 \,
    c(1-\eta+\nu)}{1+\eta-\nu}h^{1+\eta-\nu}.
\end{align*}

For the terms with $l>1$, it holds by \eqref{eq:C_hder} and \eqref{eq:C_hni}, that
\[
\norm{C_h((2l-1)h)\xi}_{\sY} \le 
    \frac{h^2\norm{C}_{\mathcal{L}(\dom((-A)^{\nu}),\sY)}c(1-\eta+\nu)}{2(2h)^{1-\eta+\nu}(l-1)^{1-\eta+\nu}}\norm{\xi}_{\dom((-A)^{\eta})}.
\]
Now summing up the bounds for $\norm{C_h((2l-1)h)\xi}_{\sY}^2$ for $l=1,...,2^{K-1}n$ and using $\sum_{l=2}^{\infty} \frac{1}{(l-1)^{2+2(\nu-\eta)}} \le
\frac{2-2(\eta-\nu)}{1-2(\eta-\nu)}$ yields a bound for the sum in \eqref{eq:batch}, and finally Lemma~\ref{lma:batch} can be used to get the result with
\begin{align} \label{eq:Man}
M=&\frac{2
  \norm{C}_{\mathcal{L}(\dom((-A)^{\nu}),\sY)}^2\Ex{\norm{z(T)-\hat
        z_{T,n}}_{\Xs}^2}\Ex{\norm{x}_{\dom((-A)^{\eta})}^2}}{(2^{1+2(\eta-\nu)}-1)\min(\eig(R))}\times 
 \\ \nonumber & \qquad \times\left(
    M_{\nu,\eta} + \frac{c(1- \eta+ \nu)^2}{2^{4-2(\nu-\eta)}}\frac{2- 2(\eta-
      \nu)}{1- 2(\eta- \nu)}   \right)
\end{align}
where the term with $l=1$ gives
\begin{equation} \nonumber M_{\nu,\eta}=\left\{ \begin{array}{ll}
      \frac{4 (\ln 2)^2c(1-\eta+\nu)^2}{(1+\eta-\nu)^2} & \textrm{if
      } \eta > \nu, \vspace{2mm} \\
      \frac{2^{2+2(\eta-\nu)}c(\nu-\eta)^2}{(1+\eta-\nu)^2} &
      \textrm{if } \eta \le \nu.
   \end{array} \right.
\end{equation}
\end{proof}
\begin{exm} \label{ex:heat}   {\rm
Consider the 1D heat equation
\begin{equation} \nonumber
\begin{cases}
  \frac{\partial}{\partial t} z(x,t)=\frac{\partial^2}{\partial x^2}z(x,t), \quad x \in [0,1], \\
  z(0,t)=z(1,t)=0, \\
  z(x,0)=z_0, \\
  dy(t)=\frac{\partial}{\partial x}z(0,t) \, dt + dw(t)
\end{cases}
\end{equation}
with state space $\sX=L^2(0,1)$ and $\dom(A)=H_0^2[0,1]$. Assume $z_0
\in \dom(A)$ almost surely. Now the spectrum of $A$ is $\{ -\pi^2k^2
\}$ and the corresponding eigenvectors are $e_k=\sin(\pi k x)$. Then
it is easy to see that the assumptions of Theorem~\ref{thm:unbounded}
are satisfied with $\delta=2$ and $\gamma=1/2$ and thus the theorem
implies convergence rate $\mathcal{O}(n^{-1/2})$ for
$\Ex{\norm{\hat{z}_{T,n}-\hat z(T)}_{\sX}^2}$. The assumptions of Theorem~\ref{thm:unbounded2} are satisfied with $\delta=2$ implying convergence rate $\mathcal{O}(n^{-3/4})$, and Theorem~\ref{thm:analytic1} clearly implies convergence rate $\mathcal{O}(n^{-1})$ but we can do even better.

Denoting $z=\sumk \alpha_k e_k$ we have
$\norm{z}_{\dom((-A)^{\nu})}^2=\sumk k^{4\nu}\alpha_k^2$. For the
output it holds that
\begin{equation} \nonumber |Cz|^2=\left|\sumk\pi k \alpha_k \right|^2 \le
  \pi {\sumk \frac1{k^{1+\epsilon}}} \, {\sumk k^{3+\epsilon}\alpha_k^2}
\end{equation}
from which it can be deduced that $C \in
\mathcal{L}(\dom((-A)^{\nu}),\sY))$ for $\nu >
3/4$. 
Now Theorem~\ref{thm:analytic2} implies convergence rate
$\mathcal{O}(n^{-3/2+\epsilon})$ for $\Ex{\norm{\hat{z}_{T,n}-\hat z(T)}_{\sX}^2}$ with $\epsilon
>0$ --- of course, with a multiplicative constant that tends to
infinity as $\epsilon \to 0$.  }
\end{exm}

\section{Input noise} \label{sec:noise}

In this section, we shall study the effect of the input noise $u$ in \eqref{eq:system}. This is done only under the assumption of an admissible observation operator $C$. The main theorem is a generalization of \cite[Theorem~2]{tempI} where $C$ was assumed to be bounded, and the proof follows the same outline.

In the cases without input noise, the state $z(t)$  was parameterized by the initial state through $z(t)=e^{At}x$. Now we define the solution operator $S(t)$ through
\[
S(t):[x,u] \mapsto e^{At}x + \int_0^t e^{A(t-s)}Bdu(s).
\]
Formally we first define $S(t)$ for $u \in H^1(0,T;\sU)$ and then extend it for Brownian motion as a Wiener integral. Then the solution to \eqref{eq:system} is given by $z(t)=S(t)[x,u]$ and the conditional expectation over a given sigma algebra $\sigma$ by $\Ex{z(t)|\sigma}=S(t)\Ex{[x,u]|\sigma}$. In the following theorem, we consider virtually estimating $[x,u]$ using the outputs $y(t_j)$ and then the state estimate is obtained by $S(T)$. 
The well-posedness of $S(t)[0,\tilde u_j]$ for $\tilde u_j := \Ex{u|\{y(t_i), i=1,...,j\}}$ is established next.
\begin{lma}
For fixed $j$, it holds that $\tilde u_j(\cdot) \in H^1(0,T;\sU)$.
\end{lma}
\begin{proof}
For fixed $t \in [0,T]$ and $j \in \mathbb{N}$, it holds that
\[
\Ex{u(t)|y(\textup{T}_j)}=\Cov{u(t),y(\textup{T}_j)}\Cov{y(\textup{T}_j),y(\textup{T}_j)}^{-1}y(\textup{T}_j).
\]
Here $\Cov{y(\textup{T}_j),y(\textup{T}_j)}^{-1}y(\textup{T}_j)$ is a well defined finite-dimensional vector, and so we only need to differentiate $\Cov{u(t),y(\textup{T}_j)} \in \mathcal{L}(\sY^j,\sU)$ which is equivalent to the differentiability of its adjoint, $\Cov{y(\textup{T}_j),u(t)}$ which is a block operator with components $\Cov{y(t_i),u(t)}$, for $i=1,...,j$. We show that each block has a bounded derivative.

Since
\begin{align*}
y(t_i)&=C\int_0^{t_i}e^{As}x \, ds + C \int_0^{t_i} \!\! \int_0^s e^{A(s-r)}B \, du(r) \, ds + w(t_i)
\\ &=C\int_0^{t_i}e^{As}x \, ds+C\int_0^{t_i} \int_r^{t_i} e^{A(s-r)}Bds \, du(r) +w(t_i)
\end{align*}
it holds that
\[
\Cov{y(t_i),u(t)}=C\int_0^{t \wedge t_i} \int_r^{t_i}e^{A(s-r)}BQ ds \, dr
\]
and therefore
\[
\frac{d}{dt}\Cov{y(t_i),u(t)}=\left\{ \!\! \begin{array}{ll} \int_t^{t_i}Ce^{A(s-t)}BQ ds, & \textup{if } t < t_i, \\ 0, & \textup{if } t \ge t_i \end{array} \right.
\]
which is  bounded for $t,t_i \in [0,T]$, concluding the result.
\end{proof}

We are now ready to proceed to the main result of the paper.

\begin{thm} \label{thm:noise}
Assume that $C$ is an admissible observation operator, that is, for all $T\ge0$ there exists $H_T\ge0$ such that $\int_0^T \norm{Ce^{At}x}_{\Ys}^2 dt \le H_T^2 \norm{x}_{\Xs}^2$. In addition, assume that one set of assumptions in Theorems~\ref{thm:unbounded}--\ref{thm:analytic2} are satisfied. Then for $\hat z_{T,n}$ and $\hat z(T)$ defined in \eqref{eq:estimates} it holds that
\[
\Ex{\norm{\hat z_{T,n}-\hat z(T)}_{\Xs}^2} \le \frac{M_1(T)}n+\frac{M_2(T)}{n^2} + \textup{err}_x
\]
where $M(T)$ is given below in \eqref{eq:M_input} and $\textup{err}_x$ is given by the respective Theorem~\ref{thm:unbounded}--\ref{thm:analytic2}.

\end{thm}
\begin{proof}
The first part of the proof follows essentially the same outline as the proof of Lemma~\ref{lma:batch}.
Say we are estimating $[x,u]$ and we have $[\tilde x_j,\tilde u_j] := \Ex{[x,u]|\mathcal{F}_j}$ and the corresponding error covariance $\mathbb{P}_j$. Then the state estimate and the corresponding error covariance are given by $\tilde z_j = S(T)[\tilde x_j,\tilde u_j]$ and $S(T)\mathbb{P}_jS(T)^*$ --- although in the last equation the formal adjoint $S(T)^*$ is only defined in connection  with the covariance $\mathbb{P}_j$, namely $\mathbb{P}_jS(T)^*h=\Ex{[\tilde x_j-x,\tilde u_j-u]\ip{S(T)[\tilde x_j-x,\tilde u_j-u],h}_{\sX}}$.

The output in \eqref{eq:system} is given by
\[
y(t)=C\int_0^{t}e^{As}x \, ds + C \int_0^{t} \!\! \int_0^s e^{A(s-r)}B \, du(r) \, ds + w(t).
\]
As with the noiseless case, when we are including the measurement $y(t_{j+1})$, in order to get rid of the output noise correlation,  we shall subtract the linear interpolant from $y(t_{j+1})$, namely define
\[
\tilde y_{j+1} = y(t_{j+1})-\frac12y(t_{j+1}+h)-\frac12y(t_{j+1}-h)
\]
where $h=\frac{T}{2^Kn}$ for the corresponding $K$, given in \eqref{eq:times}. Now this output can be written as
\[
\tilde y_{j+1} = \Ch(t_{j+1})[x,u]^T+\tilde w_{j+1}
\]
where $\tilde w_{j+1} \sim N(0,h/2 \, R)$ is independent of $w(t_i)$ and hence, $\tilde w_i$ with $i=1,...,j$, and $\Ch (t):=[C_h(t), C_{h,u}(t)]$ with $C_h(t)$ defined in \eqref{eq:C_h} and
\begin{equation} \label{eq:Chu}
C_{h,u}(t)u:=\frac{C}{2} \! \left( \int_{t-h}^t \int_0^s \! e^{A(s-r)}Bdu(r) \, ds - \! \int_t^{t+h} \!\! \int_0^s \! e^{A(s-r)}Bdu(r) \, ds \! \right) \hspace{-5mm}
\end{equation}
for $t \ge h$. Now the error covariance increment is obtained just as in Section~\ref{sec:noiseless}, and by the same computation as in the proof of Lemma~\ref{lma:batch} (just replacing $e^{AT}$ by $S(T)$, $P_j$ by $\mathbb{P}_j$, and $C_h$ by $\Ch$), we get
\begin{equation} \label{eq:bd_inc}
\Ex{\norm{\tilde z_{j+1}-\tilde z_j}_{\sX}^2 } \le \frac{2}{h\min(\eig(R))} \Ex{\norm{\Ch [x,u]}_{\sY}^2} \Ex{\norm{\hat z_{T,n}-z(T)}_{\sX}^2}. \hspace{-2mm}
\end{equation}
In order to get a suitable bound for the increment, we must find a bound for the term $\Ex{\norm{\Ch [x,u]}_{\sY}^2} =  \Ex{\norm{C_h x}_{\sY}^2}+\Ex{\norm{C_{h,u} u}_{\sY}^2}$ (recall that $x$ and $u$ are independent). The first term is now bounded as in one of the Theorems~\ref{thm:unbounded}--\ref{thm:analytic2}, so then remains the input noise induced term. To evaluate $C_{h,u}(t_{j+1})u$, note that
\begin{align*}
\int_0^t e^{A(t-s)}Bdu(s)=&\int_0^{t_{j+1}} \! e^{A(t-s)}Bdu(s) \\ & +\int_{t_{j+1}}^t \! A \int_0^s e^{A(s-r)}Bdu(r) \, ds + \int_{t_{j+1}}^t B du(s)
\end{align*}
for $t \ge t_{j+1}$. In order to get a similar expression for $t < t_{j+1}$, just change $t_{j+1} \leftrightarrow t$ in the bounds of the last two integrals and put minus signs in front of them. Of course the last term is just $\int_{t_{j+1}}^t B du(s)=B(u(t)-u(t_{j+1}))$. Inserting this to \eqref{eq:Chu} gives
\begin{align}
  \label{eq:Chu_mod}
&C_{h,u}(t_{j+1})u \\ &=-\frac{C}{2} \Bigg[  \int_{t_{j+1}-h}^{t_{j+1}} \left( \int_t^{t_{j+1}} \!\! A \int_0^s e^{A(s-r)}Bdu(r) \, ds - B\big(u(t)-u(t_{j+1})\big) \right) dt \nonumber \\ \nonumber
& \qquad + \int_{t_{j+1}}^{t_{j+1}+h} \left( \int_{t_{j+1}}^t \!\! A \int_0^s e^{A(s-r)}Bdu(r) \, ds + B\big(u(t)-u(t_{j+1})\big) \right) dt \Bigg]. 
\end{align}
These two terms are very similar by nature so it suffices to find a bound for one of them and 
use the same bound for both terms. Thus, let us consider the first part of the latter term, namely
\begin{align*}
&\frac{C}{2}\int_{t_{j+1}}^{t_{j+1}+h} \!\! \int_{t_{j+1}}^t \!\! A \int_0^s e^{A(s-r)}Bdu(r) \, ds \, dt \\ &=\frac{C}{2}\int_{t_{j+1}}^{t_{j+1}+h} \!\! \int_{t_{j+1}}^t \!\! A \int_0^{t_{j+1}} e^{A(s-r)}Bdu(r) \, ds \, dt \\ & \quad +\frac{C}{2}\int_{t_{j+1}}^{t_{j+1}+h} \!\! \int_{t_{j+1}}^t \!\! A \int_{t_{j+1}}^s e^{A(s-r)}Bdu(r) \, ds \, dt \\
& = \frac{C}{2}\int_0^{t_{j+1}} \!\! \int_{t_{j+1}}^{t_{j+1}+h} \!\! (t_{j+1}+h-s)Ae^{A(s-r)}Bds \, du(r)  \\ & \quad+ \frac{C}{2}\int_{t_{j+1}}^{t_{j+1}+h} \!\! \int_r^{t_{j+1}+h} \!\! ({t_{j+1}}+h-s)Ae^{A(s-r)}Bds \, du(r) \\
&=(I) + (II).
\end{align*}
Then
\begin{align*}
\Cov{(I),(I)}=\frac14\int_0^{t_{j+1}} \!\! \int_{t_{j+1}}^{t_{j+1}+h} \!\! \int_{t_{j+1}}^{t_{j+1}+h} \!\!\!  & (t_{j+1}+h-s)(t_{j+1}+h-r) \\ & \hspace{-8mm} \times Ce^{A(s-t)}ABQB^*A^*e^{A^*(r-t)}C^*dr \, ds \, dt,
\end{align*}
and from this, 
\begin{align*}
& \Ex{\norm{(I)}_{\sY}^2}=\tr\big(\Cov{(I),(I)}\big) \\ & \le \frac{\tr(Q)}4 \int_0^{t_{j+1}}  \!\! \left( \int_{t_{j+1}}^{t_{j+1}+h} \!\!\! (t_{j+1}+h-s) \bnorm{Ce^{A(s-t)}AB}_{\mathcal{L}(\sU,\sY)} ds \right)^2  dt
\\
& \le  \frac{h^3}{12} \tr(Q) \int_0^{t_{j+1}} \int_{t_{j+1}}^{t_{j+1}+h} \bnorm{Ce^{A(s-t)}AB}_{\mathcal{L}(\sU,\sY)}^2 ds \, dt
\end{align*}
where the second inequality follows by Cauchy--Schwartz inequality. Now make a change of variables $\tau = s-t$ and note that for a fixed $s\in [t_{j+1},t_{j+1}+h]$, it holds that $\tau \in [s-t_{j+1},s] \subset [0,T]$, yielding
finally
\[
\Ex{\norm{(I)}_{\sY}^2} \le \frac{h^4}{12} \tr(Q) \norm{Ce^{As}AB}_{L^2(0,T;\mathcal{L}(\sU,\sY))}^2
\]
where $\norm{Ce^{As}AB}_{L^2(0,T;\mathcal{L}(\sU,\sY))}^2 \le H_T^2\norm{B}_{\mathcal{L}(\sU,\dom(A))}^2$  due to the admissibility assumption.

For the second term we have
\begin{align*}
\Cov{(II),(II)}=\frac14 \int_{t_{j+1}}^{t_{j+1}+h} \!\! \int_t^{t_{j+1}+h} \!\! \int_t^{t_{j+1}+h} \!\!\! & (t_{j+1}+h-s)(t_{j+1}+h-r) \\ & \hspace{-13mm} \times Ce^{A(s-t)}ABQB^*A^*e^{A^*(r-t)}C^*dr \, ds \, dt
\end{align*}
and, as with the first term,
\[
\Ex{\norm{(II)}_{\sY}^2}=\tr\big(\Cov{(II),(II)}\big) \le \frac{h^4}{12} \tr(Q)  \norm{Ce^{As}AB}_{L^2(0,T;\mathcal{L}(\sU,\sY))}^2.
\]
In $(I)$, $r \in [0,t_{j+1}]$ and in $(II)$, $r \in [t_{j+1},t_{j+1}+h]$, and thus they are independent.
Therefore,
\[
\Ex{\norm{(I)+(II)}_{\sY}^2} \le \frac{h^4}{6} \tr(Q) \norm{Ce^{As}AB}_{L^2(0,T;\mathcal{L}(\sU,\sY))}^2.
\]
It is well known that
\[
\Cov{ \int_0^h Bu(t) \,dt, \int_0^h Bu(t) \,dt}=\frac{h^3}3 BQB^*
\]
and so
\[
\Ex{\norm{\frac{C}2\int_{t_{j+1}}^{t_{j+1}+h} \!\!\!\! B(u(t)-u(t_{j+1})) dt}_{\sY}^2} \le \frac{h^3}{12}\norm{C}_{\mathcal{L}(\dom(A),\sY)}^2\norm{B}_{\mathcal{L}(\sU,\dom(A))}^2 \tr(Q).
\]
In \eqref{eq:Chu_mod}, the two $B(u(t)-u(t_{j+1}))$-terms are independent (because in the first one, $t \le t_{j+1}$ and in the second, $t \ge t_{j+1}$) and by utilizing this and gathering the above bounds, we get
\begin{align*}
&\Ex{\norm{C_{h,u}(t_{j+1})u}_{\sY}^2} \le 6\Ex{\norm{(I)+(II)}_{\sY}^2}+6  \frac{h^3}{12}\norm{C}_{\mathcal{L}(\dom(A),\sY)}^2\norm{B}_{\mathcal{L}(\sU,\dom(A))}^2 \tr(Q) \\
& \quad \le h^4 \norm{Ce^{As}AB}_{L^2(0,T;\mathcal{L}(\sU,\sY))}^2\tr(Q)+\frac{h^3}{2}\norm{C}_{\mathcal{L}(\dom(A),\sY)}^2\norm{B}_{\mathcal{L}(\sU,\dom(A))}^2 \tr(Q).
\end{align*}
Combining this with \eqref{eq:bd_inc}  gives
\begin{align*}
&\Ex{\norm{\tilde z_{j+1}-\tilde z_j}_{\sX}^2 } \\ & \le
 \frac{2h^3}{\min(\eig(R))}  \norm{Ce^{As}AB}_{L^2(0,T;\mathcal{L}(\sU,\sY))}^2\tr(Q)\Ex{\norm{\hat z_{T,n}-z(T)}_{\sX}^2} \\ & \quad +\frac{h^2}{\min(\eig(R))}\norm{C}_{\mathcal{L}(\dom(A),\sY)}^2\norm{B}_{\mathcal{L}(\sU,\dom(A))}^2 \tr(Q)\Ex{\norm{\hat z_{T,n}-z(T)}_{\sX}^2} + \textup{incr}_{[x,j+1]}
 \end{align*}
where $\textup{incr}_{[x,j+1]}$ is the contribution of the initial state $x$ obtained from one of the Theorems~\ref{thm:unbounded}--\ref{thm:analytic2}.

As was done in the proof of Theorem~\ref{thm:unbounded}, the bound obtained for increment $\Ex{\norm{\tilde z_{j+1}-\tilde z_j}_{\sX}^2}$ is multiplied by $\frac{T}{2h}$ to get an upper bound for increments corresponding to one value $h$, that is, one $K$ (see Figure~\ref{fig:times}):
\begin{equation} \label{eq:bdK}
\sum_{j=2^{K-1}n}^{2^Kn-1} \Ex{\norm{\tilde z_{j+1}-\tilde z_j}_{\sX}^2} \le \frac{M_1(T)}{2^Kn}+\frac{3M_2(T)}{(2^Kn)^2}+ \textup{incr}_{[x,j+1]}
\end{equation}
with
\begin{equation} \label{eq:M_input}
\begin{cases}
\! M_1(T)=\frac{T^2\tr(Q)}{2\min(\eig(R))}\! \norm{C}_{\mathcal{L}(\dom(A),\sY)}^2\norm{B}_{\mathcal{L}(\sU,\dom(A))}^2 \! \Ex{\! \norm{\hat z_{T,n}-z(T)}_{\sX}^2}. \hspace{-2mm} \\
\! M_2(T)=\frac{T^3\tr(Q)}{3\min(\eig(R))}  H_T^2\norm{B}_{\mathcal{L}(\sU,\dom(A))}^2\Ex{\norm{\hat z_{T,n}-z(T)}_{\sX}^2}, 
\end{cases}
\end{equation}
Now summing up \eqref{eq:bdK} for $K=1,2,...$ yields the result. 
\end{proof}

\section{Discussion}

In this paper we extended the convergence results presented by the author in \cite{tempI}. There the convergence rate estimates for $\Ex{\norm{\hat z_{T,n}-\hat z(T)}_{\sX}^2}$ were shown for finite-dimensional systems and infinite-dimensional systems with bounded observation operator $C$. Now convergence rate estimates were found for systems with unbounded observation operators with some additional assumptions on the system operators. Firstly, a result was shown for systems with diagonalizable main operators $A$. In this case, some additional assumptions were needed, including a slightly nonstandard assumption on the output operator (assumption {\it (iii)} in Theorem~\ref{thm:unbounded}). In the problems arising from PDEs on one-dimensional spatial domains, this is not a big problem but unfortunately with more complicated systems, finding a suitable $\gamma$ might be close to a mission impossible. The spectral asymptotics, on the other hand, is an extensively studied field --- so much so that it has even been a subject of a few books, such as \cite{Spectral1} by Levendorski\`i and \cite{Spectral2} Safarov and Vassiliev. Theorem~\ref{thm:unbounded2} treats the case with assuming essentially just the admissibility of the observation operator $C$. Two results were shown for systems with analytic semigroups, in which case the other technical assumptions were not needed. 
The effect of the input noise was studied in Theorem~\ref{thm:noise}, which extends the corresponding earlier result \cite[Theorem~2]{tempI} to admissible observation operators.

In all results of the paper --- except for the analytic semigroup case without input noise --- the convergence rate estimates are of the form $\frac{M T^{k+1}}{n^k}=MT \Delta t^k$, meaning that the estimates deteriorate as $T$ grows. In general, this cannot be completely avoided because there is no guarantee that the output sampling does not cause essential loss of information. A result where the bound would not deteriorate as $T$ grows could be possible under some additional assumptions (like exponential stability of the system), but the long time behaviour should be anyway studied by comparing the solutions of the corresponding discrete- and continuous-time algebraic Riccati equations.

We remark that in this paper, as well as in \cite{tempI}, it has been assumed that the time discretization can be done perfectly and the only error source is the time-sampling of the output signal $y$ that is defined in continuous time. Further research would be needed to estimate the error caused by approximate discretization schemes.


\end{document}